\documentclass[11pt]{article}
\usepackage[utf8]{inputenc}
\usepackage{amsmath,amsthm,amsfonts,amssymb,subcaption}
\usepackage{caption}
\usepackage{graphicx}
\usepackage{tikz}
\newtheorem{theorem}{Theorem}

\newtheorem{conj}{Conjecture}
\newtheorem{proposition}{Proposition}

\newtheorem*{conj*}{LCGD Conjecture}
\title{Genus Polynomials of Cubic Graphs with Non-Real Roots}

\author{MacKenzie Carr\thanks{Supported by the Natural Sciences and Engineering Research Council of Canada (NSERC) Canadian Graduate Scholarship 456422823. Email: {\tt mackenzie.carr@torontomu.ca}}
\qquad
Varpreet Dhaliwal\thanks{Supported by the NSERC Undergraduate Student Research Award 521692. Email: {\tt varpreet\_dhaliwal@sfu.ca}}
\qquad
Bojan Mohar\thanks{Supported in part by the NSERC Discovery Grant R611450 (Canada) and by the Research Project J1-2452 of ARRS (Slovenia). On leave from IMFM, Jadranska 19, 1000 Ljubljana. Email: {\tt mohar@sfu.ca}}\\[3mm]
Department of Mathematics\\
Simon Fraser University\\
Burnaby, BC~~V5A~1S6, Canada}

\date{\today}

\begin{document}

\maketitle

\begin{abstract}
Given a graph $G$, its genus polynomial is $\Gamma_G(x) = \sum_{k\geq 0} g_k(G)x^k$, where $g_k(G)$ is the number of 2-cell embeddings of $G$ in an orientable surface of genus $k$. The Log-Concavity Genus Distribution (LCGD) Conjecture states that the genus polynomial of every graph is log-concave. It was further conjectured by Stahl that the genus polynomial of every graph has only real roots, however this was later disproved. We identify several examples of cubic graphs whose genus polynomials, in addition to having at least one non-real root, have a quadratic factor that is non-log-concave when factored over the real numbers. 
\end{abstract}

\section{Introduction}

Let $G$ be a connected graph. We denote by $g_k(G)$ the number of 2-cell embeddings of $G$ in an orientable surface of genus $k$. Here we count 2-cell embeddings up to combinatorial equivalence: two 2-cell embeddings are equivalent if their facial walks are the same. Equivalently, their clockwise rotation systems are the same. This is also equivalent to saying that there is an orientation-preserving homeomorphism among the surfaces that induces identity on the embedded graph.
We refer to \cite{MoharThomassen} for details. The \emph{genus polynomial} of $G$ is defined as 
$$\Gamma_G(x) = \sum_{k\geq 0} g_k(G) x^k.$$

The following conjecture was proposed in 1989 by Gross, Robbins and Tucker \cite{GrossRobbinsTucker}. It is known as the Log-Concavity Genus Distribution (LCGD) Conjecture:

\begin{conj*}
[\cite{GrossRobbinsTucker}] For every graph $G$, the genus polynomial is log-concave. 
\end{conj*}

Let us recall that the polynomial $\sum_{k\geq 0} g_k(G)x^k$ is \emph{log-concave} if the sequence $g_0(G),g_1(G),\dots$ of its coefficients is log-concave, meaning that $(g_k(G))^2 \geq g_{k-1}(G)g_{k+1}(G)$, for every $k\in \mathbb{N}$. 

The LCGD Conjecture has been confirmed for various classes of graphs and shown to be preserved under several graph amalgamation operations. However, the conjecture remains widely open for general graphs. In \cite{Stahl}, Stahl conjectured the following stronger property of genus polynomials. 

\begin{conj}
For every graph $G$, the genus polynomial $\Gamma_G(x)$ has only real roots. 
\end{conj}

Stahl went on to prove that the conjecture holds for particular classes of graphs including bouquets $B_n$, dipoles $DP_n$, and cobblestone paths, as well as offer examples of graph classes whose genus generating matrices are known but whose genus polynomials hadn't been proven to have real roots. Liu and Wang \cite{LiuWang} used one such example, the class of $W_4$-linear graphs, to disprove Stahl's conjecture, but Chen \cite{Chen} discovered an error in the genus generating matrix used in this counterexample. Chen showed that the argument of Liu and Wang does not hold using the correct genus generating matrix for $W_4$-linear graphs. In addition, Chen verified that the genus polynomials of all graphs with maximum genus 2 are real-rooted. 

Stahl's conjecture remained open until 2010, at which time Chen and Liu \cite{ChenLiu} identified two graphs, shown in Figure~\ref{fig:1}, whose genus polynomials both have non-real roots. 

\begin{figure}[ht]
    \centering
    \begin{subfigure}[t]{0.3\textwidth}
    \begin{tikzpicture}[thick,every node/.style={circle, draw=black, fill=white, inner sep=1}]
		\node  (0) at (1.5*0.30902,1.5*0.95106) {};
		\node  (1) at (1.5*-0.80902,1.5*0.58779) {};
		\node  (2) at (1.5*-0.80902,1.5*-0.58779) {};
		\node  (3) at (1.5*0.30902,1.5*-0.95106) {};
		\node  (4) at (1.5, 0) {};
		\node  (5) at (0.85*0.30902,0.85*0.95106) {};
		\node  (6) at (0.85*-0.80902,0.85*0.58779) {};
		\node  (7) at (0.85*-0.80902,0.85*-0.58779) {};
		\node  (8) at (0.85*0.30902,0.85*-0.95106) {};
		\node  (9) at (0.85*1,0) {};

		\draw (1) to (0);
		\draw (0) to (4);
		\draw (4) to (3);
		\draw (3) to (2);
		\draw (2) to (1);
		\draw (1) to (6);
		\draw (6) to (5);
		\draw (5) to (0);
		\draw (5) to (9);
		\draw (9) to (4);
		\draw (9) to (8);
		\draw (8) to (3);
		\draw (8) to (7);
		\draw (7) to (2);
		\draw (7) to (6);
\end{tikzpicture}
\end{subfigure}
\hspace{5mm}
\begin{subfigure}[t]{0.3\textwidth}
\begin{tikzpicture}[thick,every node/.style={circle, draw=black, fill=white, inner sep=1}]
		\node  (0) at (0.6*-1,0.6*2.4142) {};
		\node  (1) at (0.6*1,0.6*2.4142) {};
		\node  (2) at (0.6*-2.4142,0.6*1) {};
		\node  (3) at (0.6*-2.4142,0.6*-1) {};
		\node  (4) at (0.6*-1,0.6*-2.4142) {};
		\node  (5) at (0.6*1,0.6*-2.4142) {};
		\node  (6) at (0.6*2.4142,0.6*-1) {};
		\node  (7) at (0.6*2.4142,0.6*1) {};
		\node  (8) at (0.35*-1,0.35*2.4142) {};
		\node  (9) at (0.35*1,0.35*2.4142) {};
		\node  (10) at (0.35*2.4142,0.35*1) {};
		\node  (11) at (0.35*2.4142,0.35*-1) {};
		\node  (12) at (0.35*1,0.35*-2.4142) {};
		\node  (13) at (0.35*-1,0.35*-2.4142) {};
		\node  (14) at (0.35*-2.4142,0.35*-1) {};
		\node  (15) at (0.35*-2.4142,0.35*1) {};
	
		\draw (2) to (0);
		\draw (0) to (1);
		\draw (1) to (7);
		\draw (7) to (6);
		\draw (6) to (5);
		\draw (5) to (4);
		\draw (4) to (3);
		\draw (3) to (2);
		\draw (0) to (8);
		\draw (8) to (15);
		\draw (15) to (2);
		\draw (14) to (15);
		\draw (14) to (3);
		\draw (14) to (13);
		\draw (13) to (4);
		\draw (12) to (13);
		\draw (12) to (5);
		\draw (12) to (11);
		\draw (11) to (6);
		\draw (11) to (10);
		\draw (10) to (7);
		\draw (10) to (9);
		\draw (9) to (1);
		\draw (9) to (8);

\end{tikzpicture}
\end{subfigure}
    \caption{Two graphs whose genus polynomials have non-real roots}
    \label{fig:1}
\end{figure}
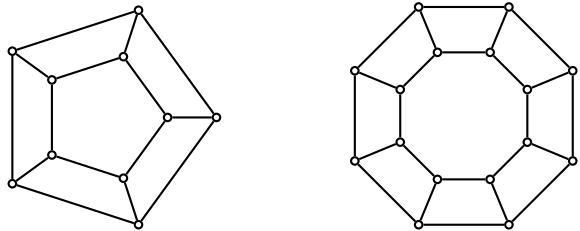

By computing the genus polynomials for all cubic graphs with up to 16 vertices, we identified a total of 226 such cubic graphs whose genus polynomials have non-real roots. The number of graphs of each order is summarized in Table 1. Since there are 40301 cubic graphs of order 18 and 510489 cubic graphs of order 20, it became infeasible to compute the number of these that have at least one non-real root. 

\begin{table}[htb]
\centering
\begin{tabular}{| c | c | c |} 
 \hline
 $n$ & \# of cubic graphs & total \# of \\ 
 ~ & with a non-real root & cubic graphs\\
 \hline
 10 & 1 & 19 \\ 
 12 & 5 & 85  \\
 14 & 26 & 509  \\
 16 & 194 & 4060  \\
 \hline
\end{tabular}
\caption{The number of cubic graphs of order $n$ whose genus polynomial has a non-real root, and the total number of cubic graphs, for $n=10,12,14,16$.}
\label{table1}
\end{table}

\section{Real Roots and Log-Concavity}

The question of whether all genus polynomials have real roots was of great interest to Stahl and others due to the connection to the log-concavity of the genus polynomial. 

Stahl's conjecture on the real-rootedness of genus polynomials, if true, would have confirmed the LCGD Conjecture of Gross, Robbins and Tucker. This is due to the log-concavity of polynomials with positive coefficients being preserved under multiplication. This fact is well known, with a nice proof given by Stanley \cite{Stanley}.

\begin{theorem}
If $A(x)$ and $B(x)$ are log-concave polynomials with non-negative coefficients and no internal zero coefficients, then so is $A(x)B(x)$.
\end{theorem}

It is clear that any linear function $ax+b$ with $a,b\in \mathbb{R}_+$, is log-concave. If a polynomial with positive coefficients is real-rooted, then its roots are nonpositive and it can be factored as $P(x) = a\prod_{j=1}^{n} (x+b_j)$, where $a>0$ and each factor $x+b_j$ is log-concave. Thus, Theorem 1 implies that $P(x)$ is log-concave. 

However, Theorem 1 implies log-concavity of some polynomials with complex roots. 

\begin{proposition}
Let $P(x)$ be a polynomial with non-negative real coefficients. If every complex root $z$ of $P$ lies in the cone $|\Re(z)| \geq \frac{1}{\sqrt{3}}|\Im(z)|$ with $\Re(z) \leq 0$, then $P(x)$ is log-concave. 
\end{proposition}

\begin{proof}
The polynomial $P(x)$ can be factored over $\mathbb{R}$ into linear and irreducible quadratic factors. Each real root is non-positive. This implies that each linear factor is log-concave. Each quadratic factor $x^2+bx+c$ has roots
$$z = \frac{-b\pm \sqrt{b^2-4c}}{2}$$
with $b^2-4c<0$. Since we have that $|\Re(z)| \geq \frac{1}{\sqrt{3}}|\Im(z)|$, it must be the case that 
$$|\Re(z)| = \frac{|b|}{2} \geq \frac{1}{\sqrt{3}}\frac{\sqrt{4c-b^2}}{2} = \frac{1}{\sqrt{3}}|\Im(z)|.$$

Note that $\Re(z) \leq 0$ guarantees that $b\geq 0$. 

Now, we have that
$$b \geq \sqrt{\frac{4c-b^2}{3}},$$
which is equivalent to $b^2 \geq c$. Therefore, each quadratic factor $x^2+bx+c$ is log-concave and $P(x)$ is a product of log-concave linear and quadratic polynomials with non-negative coefficients. By Theorem 1, $P(x)$ is log-concave. 
\end{proof}

One may be tempted to conjecture that the complex roots of genus polynomials lie in the cone indicated in Proposition 1. However, we have identified several cubic graphs, of orders 16 to 24, whose genus polynomials have roots that lie outside of the cone in Proposition 1. We include these graphs in Figure 2 and the details of the calculations in Section 3.

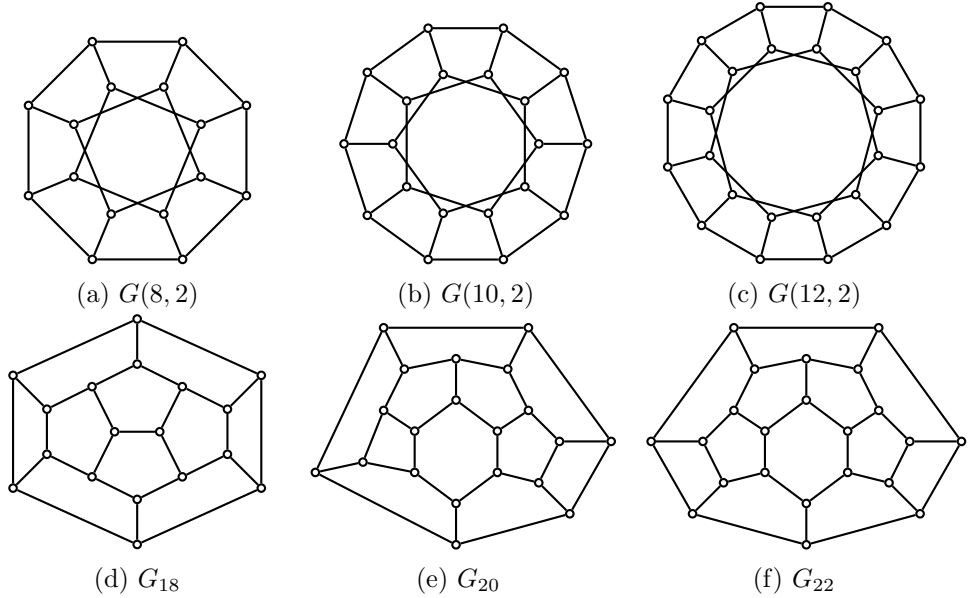
\begin{figure}[htb]
\centering
\begin{subfigure}[t]{0.31\textwidth}
\centering
\begin{tikzpicture}[thick,every node/.style={circle, draw=black, fill=white, inner sep=1}]
		\node  (0) at (0.6*-1,0.6*2.4142) {};
		\node  (1) at (0.6*1,0.6*2.4142) {};
		\node  (2) at (0.6*-2.4142,0.6*1) {};
		\node  (3) at (0.6*-2.4142,0.6*-1) {};
		\node  (4) at (0.6*-1,0.6*-2.4142) {};
		\node  (5) at (0.6*1,0.6*-2.4142) {};
		\node  (6) at (0.6*2.4142,0.6*-1) {};
		\node  (7) at (0.6*2.4142,0.6*1) {};
		\node  (8) at (0.35*-1,0.35*2.4142) {};
		\node  (9) at (0.35*1,0.35*2.4142) {};
		\node  (10) at (0.35*-2.4142,0.35*1) {};
		\node  (11) at (0.35*-2.4142,0.35*-1) {};
		\node  (12) at (0.35*-1,0.35*-2.4142) {};
		\node  (13) at (0.35*1,0.35*-2.4142) {};
		\node  (14) at (0.35*2.4142,0.35*-1) {};
		\node  (15) at (0.35*2.4142,0.35*1) {};
	
		\draw (0) to (1);
		\draw (1) to (7);
		\draw (7) to (6);
		\draw (6) to (5);
		\draw (5) to (4);
		\draw (4) to (3);
		\draw (3) to (2);
		\draw (2) to (0);
		\draw (0) to (8);
		\draw (9) to (1);
		\draw (15) to (7);
		\draw (14) to (6);
		\draw (13) to (5);
		\draw (12) to (4);
		\draw (3) to (11);
		\draw (2) to (10);
		\draw (8) to (11);
		\draw (11) to (13);
		\draw (13) to (15);
		\draw (15) to (8);
		\draw (10) to (9);
		\draw (9) to (14);
		\draw (14) to (12);
		\draw (12) to (10);
	
\end{tikzpicture}
\caption{$G(8,2)$}
\end{subfigure}
\hspace{2mm}
\begin{subfigure}[t]{0.31\textwidth}
\centering
\begin{tikzpicture}[thick,every node/.style={circle, draw=black, fill=white, inner sep=1}]
	
		\node  (0) at (-1.30902,0.951057) {};
		\node  (1) at (-0.5,1.53884) {};
		\node  (2) at (0.5,1.53884) {};
		\node  (3) at (1.30902,0.951057) {};
		\node  (4) at (1.618034,0) {};
		\node  (5) at (1.30902,-0.951057) {};
		\node  (6) at (0.5,-1.53884) {};
		\node  (7) at (-0.5,-1.53884) {};
		\node  (8) at (-1.30902,-0.951057) {};
		\node  (9) at (-1.618034,0) {};
		\node  (10) at (0.6*-1.30902,0.6*0.951057) {};
		\node  (11) at (0.6*-0.5,0.6*1.53884) {};
		\node  (12) at (0.6*0.5,0.6*1.53884) {};
		\node  (13) at (0.6*1.30902,0.6*0.951057) {};
		\node  (14) at (0.6*1.618034,0) {};
		\node  (15) at (0.6*1.30902,0.6*-0.951057) {};
		\node  (16) at (0.6*0.5,0.6*-1.53884) {};
		\node  (17) at (0.6*-0.5,0.6*-1.53884) {};
		\node  (18) at (0.6*-1.30902,0.6*-0.951057) {};
		\node  (19) at (0.6*-1.618034,0) {};
	
		\draw (0) to (1);
		\draw (1) to (2);
		\draw (2) to (3);
		\draw (3) to (4);
		\draw (4) to (5);
		\draw (5) to (6);
		\draw (6) to (7);
		\draw (7) to (8);
		\draw (8) to (9);
		\draw (9) to (0);
		\draw (0) to (10);
		\draw (10) to (12);
		\draw (12) to (2);
		\draw (12) to (14);
		\draw (14) to (4);
		\draw (14) to (16);
		\draw (6) to (16);
		\draw (16) to (18);
		\draw (18) to (8);
		\draw (18) to (10);
		\draw (19) to (9);
		\draw (19) to (11);
		\draw (11) to (1);
		\draw (11) to (13);
		\draw (13) to (3);
		\draw (13) to (15);
		\draw (15) to (5);
		\draw (15) to (17);
		\draw (17) to (7);
		\draw (17) to (19);
	
\end{tikzpicture}
\caption{$G(10,2)$}
\end{subfigure}
\hspace{2mm}
\begin{subfigure}[t]{0.31\textwidth}
\centering
\begin{tikzpicture}[thick,every node/.style={circle, draw=black, fill=white, inner sep=1}]
	
		\node  (0) at (0.9*-0.5,0.9*1.86603) {};
		\node  (2) at (0.9*0.5, 0.9*1.86603) {};
		\node  (3) at (0.9*1.36603,0.9*1.36603) {};
		\node  (4) at (0.9*1.86603,0.9*0.5) {};
		\node  (5) at (0.9*1.86603,0.9*-0.5) {};
		\node  (6) at (0.9*1.36603,0.9*-1.36603) {};
		\node  (7) at (0.9*0.5,0.9*-1.86603) {};
		\node  (8) at (0.9*-0.5,0.9*-1.86603) {};
		\node  (9) at (0.9*-1.36603,0.9*-1.36603) {};
		\node  (10) at (0.9*-1.86603,0.9*-0.5) {};
		\node  (11) at (0.9*-1.86603,0.9*0.5) {};
		\node  (12) at (0.9*-1.36603,0.9*1.36603) {};
		\node  (13) at (0.6*-0.5,0.6*1.86603) {};
		\node  (15) at (0.6*0.5, 0.6*1.86603) {};
		\node  (16) at (0.6*1.36603,0.6*1.36603) {};
		\node  (17) at (0.6*1.86603,0.6*0.5) {};
		\node  (18) at (0.6*1.86603,0.6*-0.5) {};
		\node  (19) at (0.6*1.36603,0.6*-1.36603) {};
		\node  (20) at (0.6*0.5,0.6*-1.86603) {};
		\node  (21) at (0.6*-0.5,0.6*-1.86603) {};
		\node  (22) at (0.6*-1.36603,0.6*-1.36603) {};
		\node  (23) at (0.6*-1.86603,0.6*-0.5) {};
		\node  (24) at (0.6*-1.86603,0.6*0.5) {};
		\node  (25) at (0.6*-1.36603,0.6*1.36603) {};
	
		\draw (12) to (0);
		\draw (2) to (3);
		\draw (3) to (4);
		\draw (4) to (5);
		\draw (5) to (6);
		\draw (6) to (7);
		\draw (7) to (8);
		\draw (8) to (9);
		\draw (9) to (10);
		\draw (10) to (11);
		\draw (11) to (12);
		\draw (12) to (25);
		\draw (13) to (0);
		\draw (15) to (2);
		\draw (16) to (3);
		\draw (17) to (4);
		\draw (18) to (5);
		\draw (19) to (6);
		\draw (20) to (7);
		\draw (21) to (8);
		\draw (15) to (17);
		\draw (17) to (19);
		\draw (19) to (21);
		\draw (21) to (23);
		\draw (0) to (2);
		\draw (23) to (25);
		\draw (24) to (11);
		\draw (23) to (10);
		\draw (22) to (9);
		\draw (25) to (15);
		\draw (13) to (16);
		\draw (16) to (18);
		\draw (18) to (20);
		\draw (20) to (22);
		\draw (22) to (24);
		\draw (24) to (13);
	
\end{tikzpicture}

\caption{$G(12,2)$}
\end{subfigure}

\begin{subfigure}[t]{0.31\textwidth}
\centering
\begin{tikzpicture}[thick,every node/.style={circle, draw=black, fill=white, inner sep=1}]
	\node  (0) at (0.6*-5, 0.6*-7.25) {};
		\node  (1) at (0.6*-4.25, 0.6*-8) {};
		\node  (2) at (0.6*-2.25, 0.6*-6) {};
		\node  (3) at (0.6*-2.25, 0.6*-7) {};
		\node  (4) at (0.6*-3.25, 0.6*-7.5) {};
		\node  (5) at (0.6*-2.75, 0.6*-8.5) {};
		\node  (6) at (0.6*-1.75, 0.6*-8.5) {};
		\node  (7) at (0.6*-1.25, 0.6*-7.5) {};
		\node  (8) at (0.6*0.5, 0.6*-7.25) {};
		\node  (9) at (0.6*-0.25, 0.6*-8) {};
		\node  (10) at (0.6*-0.25, 0.6*-9) {};
		\node  (11) at (0.6*-1.25, 0.6*-9.5) {};
		\node  (12) at (0.6*0.5, 0.6*-9.75) {};
		\node  (13) at (0.6*-4.25, 0.6*-9) {};
		\node  (14) at (0.6*-5, 0.6*-9.75) {};
		\node  (15) at (0.6*-3.25, 0.6*-9.5) {};
		\node  (16) at (0.6*-2.25, 0.6*-10) {};
		\node  (17) at (0.6*-2.25, 0.6*-11) {};
	
		\draw (0) to (1);
		\draw (1) to (4);
		\draw (4) to (3);
		\draw (3) to (2);
		\draw (2) to (0);
		\draw (4) to (5);
		\draw (5) to (6);
		\draw (6) to (7);
		\draw (7) to (3);
		\draw (2) to (8);
		\draw (8) to (9);
		\draw (9) to (7);
		\draw (6) to (11);
		\draw (11) to (10);
		\draw (10) to (9);
		\draw (8) to (12);
		\draw (12) to (10);
		\draw (14) to (13);
		\draw (13) to (1);
		\draw (0) to (14);
		\draw (14) to (17);
		\draw (17) to (12);
		\draw (17) to (16);
		\draw (16) to (11);
		\draw (16) to (15);
		\draw (15) to (5);
		\draw (15) to (13);
	
\end{tikzpicture}
\caption{$G_{18}$}
\end{subfigure}
\hspace{1mm}
\begin{subfigure}[t]{0.31\textwidth}
\centering
\begin{tikzpicture}[thick,every node/.style={circle, draw=black, fill=white, inner sep=1}]
	
		\node   (0) at (0.55*-1, 0.55*-1) {};
		\node  (1) at (0.55*2.5, 0.55*-1) {};
		\node  (2) at (0.55*-0.5, 0.55*-2) {};
		\node  (3) at (0.55*2, 0.55*-2) {};
		\node  (4) at (0.55*0.75, 0.55*-1.75) {};
		\node   (5) at (0.55*-1, 0.55*-3) {};
		\node  (6) at (0.55*0.75, 0.55*-2.75) {};
		\node  (7) at (0.55*-0.25, 0.55*-3.5) {};
		\node  (8) at (0.55*2.5, 0.55*-3) {};
		\node  (9) at (0.55*1.75, 0.55*-3.5) {};
		\node  (10) at (0.55*4.5, 0.55*-3.75) {};
		\node  (11) at (0.55*3.25, 0.55*-3.75) {};
		\node  (12) at (0.55*2.75, 0.55*-4.75) {};
		\node  (13) at (0.55*1.75, 0.55*-4.5) {};
		\node  (14) at (0.55*3.5, 0.55*-5.5) {};
		
		\node  (17) at (0.55*-2.65, 0.55*-4.5) {};
		\node  (18) at (0.55*-1.5, 0.55*-4.25) {};
		\node   (19) at (0.55*0.75, 0.55*-6.25) {};
		\node  (20) at (0.55*-0.25, 0.55*-4.5) {};
		\node  (21) at (0.55*0.75, 0.55*-5.25) {};
	
		\draw (0) to (1);
		\draw (1) to (3);
		\draw (3) to (4);
		\draw (4) to (2);
		\draw (2) to (0);
		\draw (2) to (5);
		\draw (5) to (7);
		\draw (7) to (6);
		\draw (6) to (4);
		\draw (3) to (8);
		\draw (8) to (9);
		\draw (9) to (6);
		\draw (1) to (10);
		\draw (10) to (11);
		\draw (11) to (8);
		\draw (11) to (12);
		\draw (12) to (13);
		\draw (13) to (9);
		\draw (10) to (14);
		\draw (14) to (12);
		\draw (17) to (0);
		\draw (5) to (18);
		\draw (19) to (14);
		\draw (19) to (17);
		\draw (17) to (18);
		\draw (7) to (20);
		\draw (18) to (20);
		\draw (19) to (21);
		\draw (21) to (13);
		\draw (21) to (20);
	
\end{tikzpicture}

\caption{$G_{20}$}
\end{subfigure}
\hspace{3mm}
\begin{subfigure}[t]{0.31\textwidth}
\centering
\begin{tikzpicture}[thick,every node/.style={circle, draw=black, fill=white, inner sep=1}]
	
		\node  (0) at (0.55*-1, 0.55*-1) {};
		\node  (1) at (0.55*2.5, 0.55*-1) {};
		\node  (2) at (0.55*-0.5, 0.55*-2) {};
		\node  (3) at (0.55*2, 0.55*-2) {};
		\node  (4) at (0.55*0.75, 0.55*-1.75) {};
		\node  (5) at (0.55*-1, 0.55*-3) {};
		\node  (6) at (0.55*0.75, 0.55*-2.75) {};
		\node  (7) at (0.55*-0.25, 0.55*-3.5) {};
		\node  (8) at (0.55*2.5, 0.55*-3) {};
		\node  (9) at (0.55*1.75, 0.55*-3.5) {};
		\node  (10) at (0.55*4.5, 0.55*-3.75) {};
		\node  (11) at (0.55*3.25, 0.55*-3.75) {};
		\node  (12) at (0.55*2.75, 0.55*-4.75) {};
		\node  (13) at (0.55*1.75, 0.55*-4.5) {};
		\node  (14) at (0.55*3.5, 0.55*-5.5) {};
		\node  (15) at (0.55*-3, 0.55*-3.75) {};
		\node  (16) at (0.55*-1.75, 0.55*-3.75) {};
		\node  (17) at (0.55*-2, 0.55*-5.5) {};
		\node  (18) at (0.55*-1.25, 0.55*-4.75) {};
		\node  (19) at (0.55*0.75, 0.55*-6.25) {};
		\node  (20) at (0.55*-0.25, 0.55*-4.5) {};
		\node  (21) at (0.55*0.75, 0.55*-5.25) {};
	
		\draw (0) to (1);
		\draw (1) to (3);
		\draw (3) to (4);
		\draw (4) to (2);
		\draw (2) to (0);
		\draw (2) to (5);
		\draw (5) to (7);
		\draw (7) to (6);
		\draw (6) to (4);
		\draw (3) to (8);
		\draw (8) to (9);
		\draw (9) to (6);
		\draw (1) to (10);
		\draw (10) to (11);
		\draw (11) to (8);
		\draw (11) to (12);
		\draw (12) to (13);
		\draw (13) to (9);
		\draw (10) to (14);
		\draw (14) to (12);
		\draw (15) to (0);
		\draw (15) to (16);
		\draw (16) to (5);
		\draw (17) to (15);
		\draw (16) to (18);
		\draw (19) to (14);
		\draw (19) to (17);
		\draw (17) to (18);
		\draw (7) to (20);
		\draw (18) to (20);
		\draw (19) to (21);
		\draw (21) to (13);
		\draw (21) to (20);
	\end{tikzpicture}
\caption{$G_{22}$}
\end{subfigure}
\caption{Six graphs whose genus polynomial has non-real roots satisfying $|\Re(z)| < |\Im(z)|/\sqrt{3}$.}
\end{figure}

\section{Examples}

The graph in Figure 2(a) is the generalized Petersen graph $G(8,2)$. Its genus polynomial is 
$$\Gamma_{G(8,2)}(x) = 39840x^4+23536x^3+2074x^2+84x+2$$
which is easily shown to be log-concave. It has two real roots, approximately equal to $-0.0572570083$ and $-0.4935182253$, and one pair of non-real roots, approximately equal to $-0.01999390944\pm 0.03710524561i$. By direct calculation, we see that 

$$\frac{1}{\sqrt{3}}(0.03710524561) \approx{} 0.02142272354 > 0.01999390944$$

Factoring $\Gamma_{G(8,2)}(x)$ over the real numbers, we get the following factors. 
\begin{equation*}
\begin{aligned}
\Gamma_{G(8,2)}(x) \approx{} & (x + 0.0572570083)(x+0.4935182253)\\
 & (x^2+0.03998781888x + 0.001776555666)
\end{aligned}
\end{equation*}

By direct computation, we can see that $$(0.03998781888)^2 \approx{} 0.00159902565
 < 0.001776555666$$ and thus the quadratic factor of $\Gamma_{G(8,2)}(x)$ above is not log-concave. 

Similar computation shows that the same is true for the other five graphs in Figure 2, which all have log-concave genus polynomials. Table 2 gives a summary of these results, including the values of the non-real roots of each genus polynomial and their corresponding non-log-concave quadratic factors, rounded to 11 decimal places. 

\begin{table}[ht]
\centering
\begin{tabular}{| c | c | c |} 
 \hline
 $G(8,2)$  & $\Gamma_{G(8,2)}(x)$ & $39840x^4+23536x^3+2074x^2+84x+2$ \\ 
 \hline
 ~ & Complex roots &  $-0.01999390944\pm 0.03710524561i$ \\
 \hline
 ~ & $|\Im(z)|/\sqrt{3}$ & $0.02142272354$  \\
 \hline
 ~ & Quadratic factor & $x^2+0.03998781888x+0.001776555666$  \\
 \hline
 \hline
 $G_{18}$ & $\Gamma_{G_{18}}(x)$ & $54272x^5 + 165824x^4+39472x^3$\\
 ~ & ~ & $+2480x^2+94x+2$ \\
 \hline
 ~ & Complex roots &  $-0.01496753672\pm 0.038599441i$ \\
 \hline
 ~ & $|\Im(z)|/\sqrt{3}$ & $0.022285398$  \\
 \hline
 ~ & Quadratic factor & $x^2+0.02993507344x+0.001713944001$  \\
 \hline
 \hline
 $G(10,2)$  & $\Gamma_{G(10,2)}(x)$ & $587040x^5 + 411400x^4+47540x^3$ \\
 ~ & ~ & $+2494x^2+100x+2$ \\ 
 \hline
 ~ & Complex roots &  $-0.00896278346 \pm 0.04522812336i$ \\
 \hline
 ~ & $|\Im(z)|/\sqrt{3}$ & $0.0261124692$  \\
 \hline
 ~ & Quadratic factor &  $x^2 + 0.01792556692x+0.00212591463$ \\
 \hline
 \hline
 $G_{20}$  & $\Gamma_{G_{20}}(x)$ & $557248x^5+431656x^4+56602x^3$ \\
 ~ & ~ & $+2964x^2+104x+2$ \\ 
 \hline
 ~ & Complex roots &  $-0.011539073495\pm 0.0389911954i$ \\
 \hline
 ~ & $|\Im(z)|/\sqrt{3}$ & $0.0225115771616$  \\
 \hline
 ~ & Quadratic factor &  $x^2+0.023078147x+0.001653463536$ \\
 \hline
 \hline
 $G_{22}$  & $\Gamma_{G_{22}}(x)$ & $692224x^6+2570304x^5+851384x^4$ \\
 ~ & ~ & $+76726x^3+3550x^2+114x+2$ \\ 
 \hline
 ~ & Complex roots &  $-0.0085736029\pm 0.03859372887i$ \\
 \hline
 ~ & $|\Im(z)|/\sqrt{3}$ & $0.02228209975$  \\
 \hline
 ~ & Quadratic factor & $x^2+0.0171472058x+0.001562982575$  \\
 \hline
 \hline
 $G(12,2)$  & $\Gamma_{G(12,2)}(x)$ & $8309664x^6 + 7244496x^5+1144338x^4$ \\
 ~ & ~ & $+75088x^3+3508x^2+120x+2$ \\ 
 \hline
 ~ & Complex roots &  $-0.002315938876 \pm 0.04585954927i$ \\
 \hline
 ~ & $|\Im(z)|/\sqrt{3}$ & $0.02647702312$  \\
  \hline
 ~ & Quadratic factor &  $x^2 + 0.004631877752x+0.002108461832$ \\
 \hline
\end{tabular}
\caption{Summary of the genus polynomials and their roots for the graphs in Figure 2.}
\label{table2}
\end{table}

\section{Conclusion}

We have identified a total of six graphs whose genus polynomials have non-real roots satisfying $|\Re(z)| < |\Im(z)|/\sqrt{3}$ and therefore have a non-log-concave quadratic factor when factored over the real numbers. However, we are limited in our ability to find infinite classes of such examples by the small number of graph classes whose genus distributions are known. One interesting feature of these graphs is that they are all planar and 3-connected, hence the constant coefficient of their genus polynomials is equal to 2. 

\section*{Acknowledgements} 

The authors would like to thank Austin Ulrigg for identifying and correcting errors in Table 1.

\bibliography{NoteRoots}

\end{document}